\documentclass[11pt,english]{article}
\usepackage[T1]{fontenc}
\usepackage[latin9]{inputenc}
\usepackage{geometry}
\usepackage{babel}
\usepackage{amsmath}
\usepackage{amsthm}
\usepackage{amssymb}
\usepackage[unicode=true,
 bookmarks=false,
 breaklinks=false,pdfborder={0 0 0},pdfborderstyle={},backref=false,colorlinks=false]
 {hyperref}
\hypersetup{
 pdfkeywords={keyword}}

\makeatletter
\numberwithin{equation}{section}
\numberwithin{figure}{section}
\theoremstyle{plain}
\newtheorem{thm}{\protect\theoremname}[section]
  \theoremstyle{definition}
  \newtheorem{defn}[thm]{\protect\definitionname}
  \theoremstyle{plain}
  \newtheorem{lem}[thm]{\protect\lemmaname}
  \theoremstyle{plain}
  \newtheorem{cor}[thm]{\protect\corollaryname}

\@ifundefined{date}{}{\date{}}
\makeatother

  \providecommand{\corollaryname}{Corollary}
  \providecommand{\definitionname}{Definition}
  \providecommand{\lemmaname}{Lemma}
\providecommand{\theoremname}{Theorem}

\begin{document}

\title{{\Large{}Constructive Harmonic Conjugates }}

\author{Mark Mandelkern}

\maketitle
{\let\thefootnote\relax}

{\let\thefootnote\relax}

\noindent {\small{}\vspace{0.8cm}
}{\small \par}
\begin{abstract}
\noindent In the study of the real projective plane, harmonic conjugates
have an essential role, with applications to projectivities, involutions,
and polarity. The construction of a harmonic conjugate requires the
selection of auxiliary elements; it must be verified, with an invariance
theorem, that the result is independent of the choice of these auxiliary
elements. A constructive proof of the invariance theorem is given
here; the methods used follow principles put forth by Errett Bishop. 
\end{abstract}
\noindent {\small{}\vspace{1.8cm}
}{\small \par}

\section{{\large{}Introduction} \label{sec 1:Introduction}}

The classical theory of the real projective plane is highly nonconstructive;
it relies heavily, at nearly every turn, on the\textit{ Law of Excluded
Middle}\emph{.} For example, in classical treatises it is assumed
that a given point is either on a given line, or outside the line,	although
this assumption is constructively invalid.\footnote{See, e.g., \cite[Example 1.1]{M16}} 

We follow the constructivist principles put forward by Errett Bishop
\cite{B67}. Avoiding the \emph{Law of Excluded Middle, }constructive
mathematics is a generalization of classical mathematics, just as
group theory, a generalization of abelian group theory, avoids the
commutative law. Thus every result and proof obtained constructively
is also classically valid. For the origins of modern constructivism,
and the disengagement of mathematics from formal logic, see Bishop's
``Constructivist Manifesto'' \cite[Chapter 1]{B67,BB85}. Further
discussion and additional references will be found in \cite{M16}. 

A constructive real projective plane is constructed in \cite{M16};
topics include Desargues's Theorem, harmonic conjugates, projectivities,
involutions, conics, Pascal's Theorem, poles, and polars. Here we
are concerned with the invariance theorem for harmonic conjugates.
The construction of the harmonic conjugate of a point requires the
selection of auxiliary elements; it must be demonstrated that the
result is uniquely determined, independent of the choice of these
auxiliary elements.

In an intuitionistic context, a harmonic conjugate construction  was
given by A. Heyting \cite[Section 7]{H28}. However, the proof of
the invariance theorem given there uses axioms for projective space;
it does not apply to a projective plane constructed using only axioms
for a plane. Moreover, the proof is incomplete; it applies only to
points distinct from the base points. For applications of harmonic
conjugates, e.g., to projectivities, involutions, and polarity, a
complete proof is required. 

A harmonic conjugate construction is given in \cite{M16}, using only
axioms for a plane; it applies uniformly to all points on the base
line. \label{error}However, the proof of the invariance theorem \cite[Theorem 4.7]{M16}
given there is incorrect; aside from the error, the proof is excessively
complicated, and objectionable on several counts. The discovery of
the error\footnote{The error in the proof of \cite[Theorem 4.7]{M16} is in the use of
the conclusion of step (7) beyond that step, whereas it is valid only
in relation to an assumption made in a previous step. } is due to Guillermo Calderón \cite{C1}; he also obtained a proof
of the invariance theorem, using the method of \cite{M16}, within
a computer formalization of projective geometry \cite{C2}. 

Using a method which is simpler, more transparent, and more direct,
than the method used in \cite{M16}, we will give a constructive proof
of the invariance theorem below in Section \ref{Sec. 3 - Inv Thm.}. 

Background information, references to other work in constructive geometry,
and  properties of the constructive real projective plane, will be
found in \cite{M16}.

\section{{\large{}Preliminaries} \label{sec 2 - Prelim}}

Axioms, definitions, and results are given in \cite{M16} for the
constructive real projective plane $\mathbb{P}$. One axiom has a
preëminent standing in the axiom system; it is indispensable for virtually
all constructive proofs involving the plane $\mathbb{P}$.
\begin{quote}
\textbf{Axiom C7.} \cite[Section 2]{M16} If $l$ and $m$ are distinct
lines, and $P$ is a point such that $P\neq l\cdot m$, then either
$P\notin l$ or $P\notin m$.
\end{quote}
This axiom is a strongly-worded constructive form of the classical
statement that the point common to two distinct lines is \emph{unique}. 

The proof of the invariance theorem requires Desargues's Theorem and
its converse. The following definition includes explicit details which
are required for constructive applications of Desargues's Theorem. 
\begin{defn}
\noindent \label{Defn Dist Triangles} Two triangles are \emph{distinct}
if corresponding vertices are distinct and corresponding sides are
distinct.\footnote{\noindent It is then easily shown that the lines joining corresponding
vertices are distinct, and the points of intersection of corresponding
sides are distinct. }

Distinct triangles are said to be \emph{perspective from the center
$O$} if the lines joining corresponding vertices are concurrent at
the point $O$, and $O$ lies outside\footnote{The relation\emph{ P lies outside l,} or \emph{l avoids P}, written
$P\notin l$, is used here in a strict, affirmative sense; q.v., \cite[Definition 2.3]{M16}. } each of the sides. 

Distinct triangles are said to be \emph{perspective from the axis
$l$} if the points of intersection of corresponding sides are collinear
on the line $l$, and $l$ avoids\footnote{Ibid.} each of the vertices. 
\end{defn}
Desargues's Theorem is adopted as an axiom in \cite{M16}, and then
used to prove the converse. \\

\noindent \textbf{Axiom. }Desargues's Theorem. \emph{If distinct triangles
are perspective from a center, then they are perspective from an axis.}
\begin{thm}
\noindent \emph{\cite[Theorem 3.2]{M16}} If distinct triangles are
perspective from an axis, then they are perspective from a center. 
\end{thm}
The preliminary results below will be required for the proof of the
invariance theorem. In classical work, harmonic conjugates are often
defined using quadrangles; this necessitates a separate definition
for the base points. Constructively, it is not known whether or not
an arbitrary point on the line coincides with a base point. Thus we
use the following less-problematic definition, which applies uniformly
to each point on the line.
\begin{defn}
\cite[Definition 4.1]{M16} \label{Defn}Let\emph{ $A$} and \emph{$B$}
be distinct points. For any point \emph{$C$} on the line $AB$, select
a line \emph{$l$} through $C$, distinct from $AB$, and select a
point \emph{$R$} lying outside each of the lines \emph{$AB$} and
$l$. Set $P=BR\cdot l$, $Q=AR\cdot l$, and \emph{$S=AP\cdot BQ$}.
Pending verification in  Theorem \ref{Main Theorem}, the point \emph{$D=AB\cdot RS$}
will be called \emph{the harmonic conjugate} \emph{of} $C$ \emph{with
respect to the points} $A,B$; we write $D=h(A,B;C)$. 
\end{defn}
\begin{lem}
\emph{\label{Lemma 42} \cite[Lemma 4.2]{M16} }In Definition \ref{Defn},
for the construction of a harmonic conjugate,

\emph{(a)} $P\neq A$, $Q\neq B$, $P\neq Q$.

\emph{(b)} $P\notin AR$, $Q\notin BR$, $A\notin BR$, $B\notin AR$.

\emph{(c)} $AR\neq BR,$ $AP\neq AR$, $AP\neq BR$, $BQ\neq BR$,
$BQ\neq AR$.
\end{lem}
\begin{lem}
\emph{\label{Lm44}\cite[Lemma 4.4]{M16}} In Definition \ref{Defn},
$h(A,B;A)=A$ and $h(A,B;B)=B$, for any selection of the auxiliary
 elements $(l,R)$. 
\end{lem}
\begin{lem}
\emph{\label{Lm45}\cite[Lemma 4.5]{M16}} In Definition \ref{Defn},
for the construction of a harmonic conjugate, 

\emph{(a)} If $C\neq A$, then $Q\notin AB$, $Q\neq S$, $S\neq A$,
and $D\neq A$.

\emph{(b)} If $C\neq B$, then $P\notin AB$, $P\neq S$, $S\neq B,$
and $D\neq B$.
\end{lem}
\begin{lem}
\emph{\label{Lm46}\cite[Lemma 4.6]{M16}} In Definition \ref{Defn},
for the construction of a harmonic conjugate, let the point $C$ be
distinct from each base point; i.e., $C\neq A$ and $C\neq B$. Then
the four points $P,Q,R,S$ are distinct and lie outside the base line
$AB$, and each subset of three points is noncollinear. 
\end{lem}

\section{The invariance theorem\label{Sec. 3 - Inv Thm.}}

In proving the invariance theorem, we consider first a situation in
which the configuration allows application of Desargues's Theorem
and its converse. 
\begin{thm}
\label{ Theorem special case}In Definition \ref{Defn}, let auxiliary
element selections $(l,R)$ and $(l',R')$ be used to construct harmonic
conjugates $D$ and $D'$ of the point $C$. If the point $C$ is
distinct from each base point; i.e., $C\neq A$ and $C\neq B$, and
\[
AR'\neq AR,\,BR'\neq BR,\,\,and\,\,l'\neq l,\,l'\neq CP_{1},\,l'\neq CQ{}_{1},
\]
 where $P_{1}=AP\cdot BR'$ and $Q_{1}=BQ\cdot AR'$, then $D=D'$. 
\end{thm}
\begin{proof}
We first check the validity of the conditions. Since $C\neq A$ and
$C\neq B$, the results of Lemma \ref{Lm46} will apply to the points
$P,Q,R,S$, and also to the points $P',Q',R',S'$. Since $R'\notin AB$,
we have $AB\neq BR'$. Since $A\neq B=AB\cdot BR'$, it follows from
Axiom C7 that $A\notin BR'$, and thus $AP\neq BR'$. By symmetry,
$BQ\neq AR'$. Thus the definitions of $P_{1}$ and $Q_{1}$ are valid.
Also, we see that $A\neq P_{1}$ and $B\neq Q_{1}$. Since $P\notin AB$,
we have $AB\neq AP$. Since $P_{1}\neq A=AB\cdot AP$, it follows
that $P_{1}\notin AB$, and thus $P_{1}\neq C$. Similarly, $Q_{1}\neq C$.
This shows that the conditions specified for $l'$ are meaningful. 

(1) Suppose that $D\neq D'$. 

(2) Since $R\neq A=AR\cdot AR'$, it follows that $R\notin AR'$,
and thus $R\neq R'$. From the conditions $AR'\neq AR$ and $BR'\neq BR$,
we see that $QR\neq Q'R'$ and $PR\neq P'R'$. By Lemma \ref{Lemma 42}(a),
we see that $l=PQ$ and $l'=P'Q'$; thus $PQ\neq P'Q'$. Since $D\neq D'=AB\cdot R'S'$,
it follows that $D\notin R'S'$, and thus $RS\neq R'S'$. Since $P'\neq C=l'\cdot CP_{1}$,
it follows that $P'\notin CP_{1}$, and thus $P'\neq P_{1}$. Since
$P'\neq P_{1}=AP\cdot BR'$, it follows that $P'\notin AP$, and thus
$P\neq P'$. Also, $AP\neq AP'$; i.e.,  $PS\neq P'S'$.\footnote{As for the necessity of taking the points $P_{1}$, $Q_{1}$ into
account, note that if the line $l'$ were to pass through the point
$P_{1}$, then we would have $PS=P'S'$. Similarly, if $l'$ were
to pass through $Q_{1}$, then we would have $QS=Q'S'$.} By symmetry, we have $Q'\notin BQ,$ $Q\neq Q'$, and $QS\neq Q'S'$.\footnote{Ibid.}
Since $S\neq A=AP\cdot AP'$, it follows that $S\notin AP'$, and
thus $S\neq S'$. 

The above, together with Lemma \ref{Lm46}, shows that the quadrangles
$PQRS$ and $P'Q'R'S'$ have distinct corresponding vertices and distinct
corresponding sides, that the corresponding contained triangles are
distinct, and that the line $AB$ avoids all eight vertices. 

(3) The triangles $PQR$ and $P'Q'R'$ have corresponding sides that
meet at points $QR\cdot Q'R'=A$, $PR\cdot P'R'=B$, and $PQ\cdot P'Q'=l\cdot l'=C$;
thus they are perspective from the axis $AB$. By the converse to
Desargues's Theorem, the triangles are perspective from a center;
setting $O=PP'\cdot QQ'$, it follows that $O\in RR'$. This also
shows that $O$ lies outside each of the six sides of these triangles,
and that $O\neq R$ and $O\neq R'$. 

(4) The triangles $PQS$ and $P'Q'S'$ are also perspective from the
axis $AB$; thus they are perspective from the center $O$, and it
follows that $O\in SS'$. Also, $O$ lies outside each of the six
sides of these triangles, $O\neq S$, and $O\neq S'$. 

(5) To apply Desargues's Theorem to the triangles $PRS$ and $P'R'S'$,
all the required distinctness conditions have been verified above,
except for one; it remains to be shown that the point $O$ lies outside
each of the sides $RS$ and $R'S'$. 

It was shown at (2) that $RS\neq R'S'$; define $E=RS\cdot R'S'$.
By cotransitivity,\footnote{For the properties of the inequality relation used here, see, e.g.,
\cite[Definition 2.1]{M16}. } either $E\neq S$ or $E\neq R$. In the first case, since $S\neq E=RS\cdot R'S'$,
it follows that $S\notin R'S'$, and thus $SS'\neq R'S'$. Since $O\neq S'$
= $SS'\cdot R'S'$, it follows that $O\notin R'S'$. In the second
case, since $R\neq E=RS\cdot R'S'$, it follows that $R\notin R'S'$,
and thus $RR'\neq R'S'$. Since $O\neq R'=RR'\cdot R'S'$, it follows
that $O\notin R'S'$. Thus in each case we obtain $O\notin R'S'$.
Similarly, either $E\neq S'$ or $E\neq R'$, and by symmetry we find
that $O\notin RS$. 

(6) Now the triangles $PRS$ and $P'R'S'$ are perspective from the
center $O$. By Desargues's Theorem, these triangles are perspective
from the axis $(PS\cdot P'S')(PR\cdot P'R')=AB$, and thus $RS\cdot R'S'\in AB$.
Hence $E=D$ and $E=D'$, contradicting our assumption at (1); thus
we have $\neg(D\neq D')$, and it follows from the tightness property
of the inequality relation\footnote{Ibid. } that $D=D'$. 
\end{proof}
\begin{thm}
\textsc{\label{Main Theorem}Invariance Theorem}\textsc{\emph{.}}
Let the projective plane \textup{$\mathbb{P}$} be such that at least
eight distinct lines\footnote{In \cite[Section 5]{M16}, Axiom E specified at least six lines through
any point. } pass through any given point. In Definition \ref{Defn}, let auxiliary
element selections $(l,R)$ and $(l',R')$ be used to construct harmonic
conjugates $D$ and $D'$ of the point $C$. Then $D=D'$; the harmonic
conjugate construction is independent of the choice of auxiliary  elements. 
\end{thm}
\begin{proof}
We construct a third selection of auxiliary elements, and then utilize
two applications of Theorem \ref{ Theorem special case}. 

(1) Suppose that $D\neq D'$. 

(2) By cotransitivity, either $A\neq D$ or $A\neq D'$; by symmetry,
it will suffice to consider the first case. Since $A\neq D=AB\cdot RS$,
it follows from Axiom C7 that $A\notin RS$, and thus $A\neq S$.
Since $A\neq S=AP\cdot BQ$, it follows that $A\notin BQ$, and thus
$A\neq Q$. Since $A\neq Q=AR\cdot l$, it follows that $A\notin l$,
and thus $A\neq C$. Similarly, $B\neq C$. Thus the point $C$ is
distinct from each base point. 

(3) Select a line $m$ through the point $A$ such that $m\neq AB$,
$m\neq AR$, and $m\neq AR'$. Select a line $n$ through the point
$B$ such that $n\neq AB$, $n\neq BR$, and $n\neq BR'$. Since $A\neq B=AB\cdot n$,
it follows that $A\notin n$, and thus $m\neq n$. Define $R''=m\cdot n$;
since $A\notin n$, we have $A\neq R''$. Since $R''\neq A=AB\cdot m$,
it follows that $R''\notin AB$. 

Since $AR''=m$, it is clear that $AR''\neq AR$, and $AR''\neq AR'$;
by symmetry, $BR''\neq BR$ and $BR''\neq BR'$. As verified in Theorem
\ref{ Theorem special case}, we may define the points $P_{1}=AP\cdot BR''$,
$Q_{1}=BQ\cdot AR''$, $P_{2}=AP'\cdot BR''$, $Q_{2}=BQ'\cdot AR''$,
and  note that these four points are  distinct from the point $C$. 

Now select a line $l''$ through the point $C$ such that $l''\neq CR''$;
$l''\neq l$, $l''\neq CP_{1}$, $l''\neq CQ{}_{1}$; and $l''\neq l'$,
$l''\neq CP_{2}$, $l''\neq CQ{}_{2}$. Since $R''\neq C=CR''\cdot l''$,
it follows that $R''\notin l''$. 

(4) The above shows that the auxiliary element selection $(l'',R'')$
satisfies the conditions for Definition \ref{Defn}, resulting in
a harmonic conjugate $D''$. This third selection $(l'',R'')$ also
satisfies the conditions of Theorem \ref{ Theorem special case},
relating it to each of the selections $(l,R)$ and $(l',R')$. Two
applications of Theorem \ref{ Theorem special case} now show that
$D''=D$ and $D''=D'$, contradicting our assumption at (1); thus
we have $\neg(D\neq D')$, and it follows that $D=D'$. 
\end{proof}
The harmonic conjugates constructed here can now be related to the
traditional quadrangle configuration.\footnote{See, e.g., \cite[Chapter IV]{VY10}.} 
\begin{cor}
\emph{\cite[Lemma 4.8]{M16}} Let $A,B,C,D$ be collinear points,
with $A\neq B$, and $C$ distinct from both points $A$ and $B$.
Then $D=h(A,B;C$) if and only if there exists a quadrangle $PQRS$,
with vertices outside the line $AB$, such that $A=PS\cdot QR$, $B=PR\cdot QS$,
$C\in PQ$, and $D\in RS$. 
\end{cor}
Additional results concerning constructive harmonic conjugates, including
applications to projectivities, involutions, and polarity, will be
found in \cite{M16}. \\

\noindent \textbf{Acknowledgments.} \label{Acknowledgments}Thanks
are due Guillermo Calderón for discovering the error in the proof
of the invariance theorem in \cite{M16}, and for the communication
\cite{C1}. \label{References}

\noindent {\small{}Department of Mathematics}{\small \par}

\noindent {\small{}New Mexico State University}{\small \par}

\noindent {\small{}Las Cruces, New Mexico 88003 USA }{\small \par}

\noindent \emph{\small{}e-mail:}{\small{} mandelkern@zianet.com, mmandelk@nmsu.edu}{\small \par}

\noindent \emph{\small{}web:}{\small{} www.zianet.com/mandelkern }\\
{\small \par}

\noindent {\small{}April 28, 2018; rev. May 9, 2018. }\\
{\small \par}

\noindent 
\end{document}